\documentclass[12pt,letterpaper]{article}

\usepackage[latin1]{inputenc}
\usepackage{subfigure}
\usepackage{calc}
\usepackage{array}
\usepackage{subfigure}
\usepackage[indent,bf]{caption}
\usepackage{rotating}
\usepackage{setspace}
\usepackage{longtable}
\usepackage{rotating}
\usepackage{amssymb, amsmath, amsfonts, epsfig}

\usepackage{amsthm}
\usepackage{epsfig,latexsym}
\usepackage{multicol}
\usepackage{pgf,tikz}
\usepackage{verbatim}
\usepackage{url}
\usetikzlibrary{arrows}

\theoremstyle{plain}

\newtheorem{thm}{Theorem}[section]

\newtheorem{lem}[thm]{Lemma}
\newtheorem{prop}[thm]{Proposition}
\newtheorem{cor}[thm]{Corollary}

\newtheorem{defn}[thm]{Definition}

\newtheorem*{123}{1-2-3 Conjecture}
\newtheorem*{12}{1-2 Conjecture}
\newtheorem*{list123}{List 1-2-3 Conjecture}
\newtheorem*{list12}{List 1-2 Conjecture}

\newtheorem*{CN}{Combinatorial Nullstellensatz}

\newcommand{\per}{\textup{per}\,}
\newcommand{\mind}{\textup{mind}}
\newcommand{\tmind}{\textup{tmind}}
\newcommand{\pind}{\textup{pind}}
\newcommand{\ch}{\textup{ch}}
\newcommand{\M}{\mathbb{M}}

\newcommand{\Z}{\mathbb{Z}}

\newcommand{\F}{\mathbb{F}}

\newcommand{\col}{\textup{\rm col}}

\newcommand{\Se}{\chi_\Sigma^e}

\newcommand{\St}{\chi_\Sigma^t}

\newcommand{\chSe}{\ch_\Sigma^e}
\newcommand{\chPe}{\ch_ \Pi ^e}

\newcommand{\chSt}{\ch_\Sigma^t}

\begin{document}
\sloppy

\title{Bounding the weight choosability number of a graph}

 \author{Ben Seamone\footnote{Dept. d'Informatique et de recherche op\'erationnelle, Universit\'e de Montr\'eal, Montr\'eal, QC, \texttt{seamone@iro.umontreal.ca}}}


\maketitle

\begin{abstract}
Let $G = (V,E)$ be a graph, and for each $e \in E(G)$, let $L_e$ be a list of real numbers.  Let $w:E(G) \to \cup_{e \in E(G)}L_e$ be an edge weighting function such that $w(e) \in L_e$ for each $e \in E(G)$, and
let $c_w$ be the vertex colouring obtained by $c_w(v) = \sum_{e \ni v}w(e)$.
 We desire the smallest possible $k$ such that, for any choice of $\{L_e \,|\, e \in E(G)\}$ where $|L_e| \geq k$ for all $e \in E(G)$, there exists an edge weighting function $w$ for which $c_w$ is proper.  The smallest such value of $k$ is the weight choosability number of $G$.

This colouring problem, introduced by Bartnicki, Grytczuk and Niwczyk (2009), is the list variation of the now famous 1-2-3 Conjecture due to Karo{\'n}ski, {\L}uczak, and Thomason (2004).
Bartnicki et al. develop a method for approaching the problem based on the Combinatorial Nullstellensatz.  Though they show that some particular classes of graphs have weight choosability number at most $3$, it was known whether their method could be extended to prove a bound which holds for all admissible graphs.
In this paper, we show that this is indeed possible, showing that
every graph is $(\Delta + d + 1)$-weight choosable, where $\Delta$ is the graph's maximum degree and $d$ is its degeneracy.
In fact, more general results on total weight choosability are provided, where one assigns weights to edges and vertices.
Improved bounds are also established for some classes of graph products.

\end{abstract}

\section{Introduction}

A graph $G = (V,E)$ will be simple and loopless unless otherwise stated.  An {\bf edge $k$-weighting}, $w$, of $G$ is a an assignment of a number from $[k] := \{1, 2,\ldots,k\}$ to each $e \in E(G)$, that is $w: E(G) \rightarrow[k]$.  Karo{\'n}ski, {\L}uczak, and Thomason \cite{KLT04} conjecture that, for every graph without a component isomorphic to $K_2$, there is an edge $3$-weighting such the function $S: V(G) \rightarrow \Z$ given by $S(v) = \sum_{e \ni v} w(e)$ is a proper colouring of $V(G)$ (in other words, any two adjacent vertices have different sums of incident edge weights).  If such a proper colouring $S$ exists, then $w$ is a {\bf vertex colouring by sums}.  Let $\Se(G)$ be the smallest value of $k$ such that a graph $G$ has an edge $k$-weighting which is a vertex colouring by sums.  A graph $G$ is {\bf nice} if it contains no component isomorphic to $K_2$.  Karo{\'n}ski, {\L}uczak, and Thomason's conjecture (frequently called the ``1-2-3 Conjecture") may be expressed as follows:

\begin{123}
If $G$ is a nice graph, then $\Se(G) \leq 3$.
\end{123}

The best known upper bound on $\Se(G)$ is due to Kalkowski, Karo{\'n}ski and Pfender \cite{KKP1}, who show that $\Se(G) \leq 5$ if $G$ is nice.

In \cite{BGN09}, Bartnicki, Grytczuk and Niwczyk consider a ``list variation'' of the 1-2-3 Conjecture.  Assign to each edge $e \in E(G)$ a list of $k$ real numbers, say $L_e$, and choose a weight $w(e) \in L_e$ for each $e \in E(G)$.  The resulting function $w: E(G) \rightarrow \cup_{e \in E(G)} L_e$ is called an {\bf edge $k$-list-weighting}.  Given a graph $G$, the smallest $k$ such that any assignment of lists of size $k$ to $E(G)$ permits an edge $k$-list-weighting which is a vertex colouring by sums is denoted $\chSe(G)$ and called the {\bf edge weight choosability number} of $G$.  The following, stronger, conjecture is proposed in \cite{BGN09}:

\begin{list123}
If $G$ is a nice graph, then $\chSe(G) \leq 3$.
\end{list123}

A similar problem to the List 1-2-3 Conjecture for graphs is solved for digraphs in \cite{BGN09}, where a constructive method is used to show that $\chSe(D) \leq 2$ for any digraph $D$.  An alternate proof which uses Alon's Combinatorial Nullstellensatz \cite{A99} may be found in \cite{Ben2}.



Another variant of the 1-2-3 Conjecture allows each vertex $v \in V(G)$ to receive a weight $w(v)$; the colour of $v$ is then $w(v) + \sum_{e \ni v} w(e)$ rather than simply $\sum_{e \ni v} w(e)$.  Such a function $w: V \cup E \rightarrow [k]$ is called a {\bf total $k$-weighting}.  The smallest $k$ for which $G$ has a total $k$-weighting that is a proper colouring by sums is denoted $\St(G)$.  A similar list generalization as above may be considered; the smallest $k$ such that the list version holds is denoted $\chSt(G)$.  The following two conjectures are posed in \cite{PW10} and \cite{PW11} \& \cite{WZ} respectively:

\begin{12}
If $G$ is any graph, then $\St(G) \leq 2$.
\end{12}

\begin{list12}
If $G$ is any graph, then $\chSt(G) \leq 2$.
\end{list12}

Though the 1-2 Conjecture remains open, Kalkowski \cite{Kal} has shown that a total weighting $w$ of $G$ which properly colours $V(G)$ by sums always exists with $w(v) \in \{1,2\}$ and $w(e) \in \{1,2,3\}$ for all $v \in V(G), e \in E(G)$.

In \cite{WZ}, Wong and Zhu study {\bf $(k,l)$-total list-assignments}, which are assignments of lists of size $k$ to the vertices of a graph and lists of size $l$ to the edges.  If any $(k,l)$-total list-assignment of $G$ permits a total weighting which is a vertex colouring by sums, then $G$ is {\bf $(k,l)$-weight choosable}.  Obviously, if a graph $G$ is $(k,l)$-weight choosable, then $\chSt(G) \leq \max\{k,l\}$.  The List 1-2 Conjecture is equivalent to the statement that every graph is $(2,2)$-choosable.  Wong and Zhu \cite{WZ} further conjecture that every nice graph is $(1,3)$-weight choosable, a strengthening of the List 1-2-3 Conjecture.  A recent important breakthrough by Wong and Zhu \cite{WZ23} shows that every graph is $(2,3)$-weight choosable and hence $\chSt(G) \leq 3$ for every graph $G$.  There is also a good deal of literature on graph classes which are $(k,l)$-weight choosable for small values of $k$ and $l$ (see \cite{BGN09}, \cite{PY12}, \cite{PW11}, \cite{WZ}, \cite{WWZ}, \cite{WYZ}).  Of particular note, it is shown in \cite{PY12} that every nice $d$-degenerate graph is $(1,2d)$-weight choosable.
However, whether or not there exists a constant $l$ such that every nice graph is $(1,l)$-weight choosable remains unsettled.

The main purpose of this paper is show how the methods used in \cite{BGN09} may be extended to obtain a bound of the $(1,l)$-weight choosability of any nice graph $G$.
%
In Chapter 2, we present a Combinatorial Nullstellensatz approach to the List 1-2-3 and List 1-2 Conjectures.
Chapter 3 contains some intermediary lemmas on matrix permanents and colouring polynomials.
Chapters 2 and 3 are largely reliant on the work of \cite{BGN09,PW11}; results are presented in near full detail, with examples, in the interest of keeping the article self-contained, and some results are generalized where necessary.
Chapter 4 contains the main result of this paper, where we show how Bartnicki et al.'s ``permanent method" may be applied to general graphs.
The bound obtained in the general case is, unfortunately, weaker than that of Pan and Yang \cite{PY12}, however we are able to obtain improved bounds for some graph products in Chapter 5.

\section{The permanent method and Alon's Nullstellensatz}\label{CN}


Let $G = (V,E)$ be a graph, with $E(G) = \{e_1, \ldots, e_m\}$ and $V(G) = \{v_1, \ldots, v_n\}$.  Associate with each $e_i$ the variable $x_i$ and with each $v_j$ the variable $x_{m+j}$.  Define two more variables for each $v_j \in V(G)$:  $X_{v_j} = \sum_{e_i \ni v_j} \, x_i$ and $Y_{v_j} = x_{m+j} + X_{v_j}$.  For an orientation $D$ of $G$, define the following two polynomials, where $l = m+n$:
	\begin{eqnarray*}
	P_D(x_1, \ldots, x_m) = \prod_{(u,v) \in E(D)}(X_v - X_u) \\
	T_D(x_1, \ldots, x_l) = \prod_{(u,v) \in E(D)}(Y_v - Y_u).
	\end{eqnarray*}
	
Let $w$ be an edge weighting of $G$.  By letting $x_i = w(e_i)$ for $1 \leq i \leq m$, $w$ is a proper vertex colouring by sums if and only if $P_D(w(e_1), \ldots, w(e_m)) \neq 0$.  A similar conclusion can be made about $T_D$ if $w$ is a total weighting of $G$.  

This leads us to the problem of determining when the polynomials $P_D$ and $T_D$ do not vanish everywhere, i.e., when there exist values of the variables for which the polynomial is non-zero.   Alon's famed Combinatorial Nullstellensatz gives sufficient conditions to guarantee that a polynomial does not vanish everywhere.

\begin{CN}[Alon \cite{A99}]
Let $\F$ be an arbitrary field, 
and let $f=f(x_1,\ldots,x_n)$ be a polynomial in $\F[x_1,\ldots,x_n]$.
Suppose the total degree of $f$ is $\sum_{i=1}^{n}t_i$, 
where each $t_i$ is a~nonnegative integer, 
and suppose the coefficient of $\prod_{i=1}^{n}x_i^{t_i}$ in $f$ is nonzero.
If $S_1,\ldots,S_n$ are subsets of $\F$ with $|S_i|>t_i$, then
there are $s_1\in S_1, s_2\in S_2,\ldots,s_n\in S_n$ so that
\[
 f(s_1,\ldots,s_n)\neq 0.
\]
\end{CN}

For a polynomial $P \in \F[x_1,\ldots,x_l]$ 
and a monomial term $M$ of $P$, let $h(M)$ be the largest exponent of any variable in $M$.  
The {\bf monomial index} of $P$, denoted $\mind(P)$, is the minimum $h(M)$ taken over all monomials of $P$.
Define the graph parameters $\mind(G) := \mind(P_D)$ and $\tmind(G) := \mind(T_D)$, where $D$ is an orientation of $G$.  
Note that, given a graph $G$ and two orientations $D$ and $D'$, $P_D(x_1, \ldots, x_l) = \pm P_{D'}(x_1, \ldots, x_l)$; a similar argument holds for $T_D$.  
The parameters $\mind(G)$ and $\tmind(G)$ are hence well-defined.
Note that, for any graph $G$, $\tmind(G) \leq \mind(G)$. 

The following lemma is obtained 
by applying the Combinatorial Nullstellensatz to $P_D$ and $T_D$:

\begin{lem}\label{mind}
Let $G$ be a graph and $k$ a positive integer.
	\begin{enumerate}
	\item \textup{(Bartnicki, Grytczuk, Niwczyk \cite{BGN09})} If $G$ is nice and $\mind(G) \leq k$, then $\chSe(G) \leq k+1$.
	\item \textup{(Przyby{\l}o, Wo\'zniak \cite{PW11})} If $\tmind(G) \leq k$, then $\chSt(G) \leq k+1$.
	\end{enumerate}
\end{lem}

The following proposition will also be useful.

\begin{prop}\label{componentbound}
If $G$ is a graph with connected components $G_1,G_2,\ldots,G_k$, then $\mind(G) = \max\{\mind(G_i) \,:\, 1 \leq i \leq k\}$.
\end{prop}

In \cite{BGN09}, Bartnicki et al. show how one may study the permanent of particular $\{-1,0,1\}$-matrices in order to gain insight on $\mind(G)$ and $\tmind(G)$.  
Let $\M(m,n)$ denote the set of all real valued matrices  with $m$ rows and $n$ columns, and $\M(m)$ denote the set of square $m \times m$ matrices.  The permanent of a matrix $A \in \M(m)$, denoted $\per{A}$,
 is calculated as follows: 
$$\per{A} = \sum_{\sigma \in S_m} \prod_{i=1}^m a_{i,\sigma(i)}.$$  

The permanent may also be defined for a general matrix $A \in \M(m,n)$ if $n \geq m$.  Let $Q_{m,n}$ denote the set of sequences of length $m$ with entries from $[n]$ which contain no repetition of elements; such sequences are also known as {\bf $m$-permutations} from $[n]$.  For example, $Q_{2,3} = \{ (1,2), (1,3 ), (2,1 ), (2,3 ), (3,1), (3,2)\}$.  The permanent of $A$ is defined as follows:
$$\per{A} = \sum_{\alpha \in Q_{m,n}} \prod_{i=1}^m a_{i,\alpha(i)} = \sum_{i = 1}^{{n \choose m}} \per{B_i},$$
where $\{B_i \,|\, 1 \leq i \leq {n \choose m}\}$ is the set of all $m \times m$ submatrices of $A$. 




The {\bf permanent rank} of a matrix $A$ (not necessarily square) is the size of the largest square submatrix of $A$ having nonzero permanent.  Let $A^{(k)} = [A A \cdots A]$ \label{kcopydefn} denote the matrix formed of $k$ consecutive copies of $A$.  If $A$ has size $m \times l$, then the {\bf permanent index} of $A$ is the smallest $k$, if it exists, such that $A^{(k)}$ has permanent rank $m$.  This parameter is denoted $\pind(A)$. If such a $k$ does not exist, then $\pind(A) := \infty$.  Alternately, $\pind(A)$ is the smallest $k$ such that a square matrix of size $m$ having nonzero permanent can be constructed by taking columns from $A$, each column taken no more than $k$ times.

There are three matrices related to directed graphs which will be of interest:
%

\begin{defn}
Let $G = (V,E)$ be a graph, $V(G) = \{v_1, \ldots, v_n\}$, $E(G) = \{e_1, \ldots, e_m\}$.  For an orientation $D$ of $G$, define the matrices $A_D \in \M(m)$, $B_D \in \M(m,n)$, and $M_D \in \M(m, m+n)$ as follows:
	\begin{itemize}
	\item $A_D = (a_{i,j})$ where $a_{i,j} =  \left\{\begin{array}{ll}
					1 & \textup{if $e_j$ is incident with the head of $e_i$} \\
					-1 & \textup{if $e_j$ is incident with the tail of $e_i$} \\
					0 & \textup{otherwise}
					\end{array} \right.$
	\item $B_D = (b_{i,j})$ where $b_{i,j} =  \left\{\begin{array}{ll}
					1 & \textup{if $v_j$ is the head of $e_i$} \\
					-1 & \textup{if $v_j$ is the tail of $e_i$} \\
					0 & \textup{otherwise}
					\end{array} \right.$
	\item $M_D = \left( A_D \mid B_D\right)$.
	\end{itemize}
\end{defn}

%

%

%

%

%

%

%

%


The following lemmas, which relate the matrices $A_D$, $B_D$, and $M_D$ to the polynomials $P_D$ and $T_D$, provide the fundamental link between the graphic polynomials of interest and matrix permanents:

\begin{lem}[Bartnicki, Grytczuk, Niwczyk \cite{BGN09}]\label{coefficientrelation}
Let $A = (a_{ij}) \in \M(m)$ have finite permanent index.  If $P(x_1, \ldots, x_m) = \prod_{i = 1}^m(a_{i1}x_1 + \ldots + a_{im}x_m)$, then $\mind(P) = \pind(A_D)$.
\end{lem}

The proof is omitted, but the result follows from the fact that the coefficient of $x_1^{k_1}x_2^{k_2}\cdots x_m^{k_m}$ in the expansion of $P$ is equal to $\frac{\per(M)}{k_1!\cdots k_m!}$ where $M$ is the $m \times m$ matrix where column $a_j$ from $A$ appears $k_j$ times.  
Lemma \ref{coefficientrelation} immediately implies the following vital link between the (total) monomial index of a graph $G$ and the permanent index of $A_D$ (respectively, $T_D$) for any orientation $D$ of $G$:

\begin{lem}\label{mindtopind}
Let $D$ be an orientation of a graph $G$.
	\begin{enumerate}
	\item \textup{(Bartnicki, Grytczuk, Niwczyk \cite{BGN09})} If $G$ is nice, then $\mind(G) = \pind(A_D)$.
	\item \textup{(Przyby{\l}o, Wo\'zniak \cite{PW11})} For any graph $G$, $\tmind(G) = \pind(M_D)$. 
	\end{enumerate}
\end{lem}


Lemmas \ref{mind} and \ref{mindtopind} imply the following relationship between $\chSe(G)$ and $A_D$, and $\chSt(G)$ and $T_D$:

\begin{cor}\label{pind}
Let $G$ be a graph, $D$ an orientation of $G$, and $k$ a positive integer.
	\begin{enumerate}
	\item If $G$ is nice and $\pind(A_D) \leq k$, then $\chSe(G) \leq k+1$.
	\item If $\pind(M_D) \leq k$, then $\chSt(G) \leq k+1$.
	\end{enumerate}
\end{cor}

Consider the following illustrative example.  Let $D$ be the digraph in Figure \ref{matrixex} and let $G$ be its underlying simple graph. 
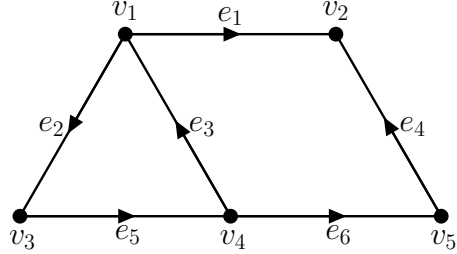
\begin{figure}[h!]
\centering
\scalebox{0.7}{
\begin{tikzpicture}
\clip(-3,-0.7) rectangle (7,4.2);
\draw [line width=1.3pt] (0,3.46)-- (-2,0);
\draw [line width=1.3pt] (0,3.46)-- (4,3.46);
\draw [line width=1.3pt] (2,0)-- (0,3.46);
\draw [line width=1.3pt] (6,0)-- (4,3.46);
\draw [line width=1.3pt] (-2,0)-- (2,0);
\draw [line width=1.3pt] (2,0)-- (6,0);
\draw [-triangle 45, line width=1.2pt] (0,3.46) -- (-1.09,1.57);
\draw [-triangle 45, line width=1.2pt] (2,0) -- (0.95,1.82);
\draw [-triangle 45, line width=1.2pt] (0,3.46) -- (2.2,3.46);
\draw [-triangle 45, line width=1.2pt] (6,0) -- (4.91,1.88);
\draw [-triangle 45, line width=1.2pt] (2,0) -- (4.2,0);
\draw [-triangle 45, line width=1.2pt] (-2,0) -- (0.2,0);
\fill [color=black] (-2,0) circle (4.0pt);
\draw[color=black] (-1.95,-0.45) node {\Large $v_3$};
\fill [color=black] (2,0) circle (4.0pt);
\draw[color=black] (2.05,-0.45) node {\Large $v_4$};
\fill [color=black] (0,3.46) circle (4.0pt);
\draw[color=black] (0,3.9) node {\Large $v_1$};
\fill [color=black] (6,0) circle (4.0pt);
\draw[color=black] (6.05,-0.45) node {\Large $v_5$};
\fill [color=black] (4,3.46) circle (4.0pt);
\draw[color=black] (4,3.9) node {\Large $v_2$};
\draw[color=black] (-1.4,1.7) node {\Large $e_2$};
\draw[color=black] (2,3.8) node {\Large $e_1$};
\draw[color=black] (1.46,1.7) node {\Large $e_3$};
\draw[color=black] (5.46,1.7) node {\Large $e_4$};
\draw[color=black] (0.05,-0.35) node {\Large $e_5$};
\draw[color=black] (4.05,-0.35) node {\Large $e_6$};
\end{tikzpicture}
}
\caption{A digraph used to illustrate $A_D$, $B_D$, and $M_D$}\label{matrixex}
\end{figure}
\\
The associated polynomial, $P_D$, is
	\begin{small}
	\begin{align*}
	P_D(x_1, \ldots, x_6) =\, &(x_1 + x_4 - x_1 - x_2 - x_3)\times(x_2 + x_5 - x_1 - x_2 - x_3) \\
	& \times (x_1 + x_2 + x_3 - x_3 - x_5 - x_6) \times (x_1 + x_4 - x_4 - x_6) \\
	& \times (x_3 + x_5 + x_6 - x_2 - x_5) \times(x_4 + x_6 - x_3 - x_5 - x_6) \\
	=\, &(x_4 - x_2 - x_3) \times(x_5 - x_1 - x_3) \times(x_1 + x_2 - x_5 - x_6) \\
	& \times (x_1 - x_6) \times(x_3 + x_6 - x_2) \times(x_4 - x_3 - x_5).
	\end{align*}
	\end{small}
\indent Recalling the definition of $M_D$, note that the coefficients of each factor in $P_D$ correspond to the entries in each row of $A_D$:
\[M_D = [A_D\,|\,B_D] =
\left\lgroup \begin{array}{cccccc|ccccc}
0 & -1 & -1 & 1 & 0 & 0 & -1 & 1 & 0 & 0 & 0 \\
-1 & 0 & -1 & 0 & 1 & 0 & -1 & 0 & 1 & 0 & 0 \\
1 & 1 & 0 & 0 & -1 & -1 & 1 & 0 & 0 & -1 & 0 \\
1 & 0 & 0 & 0 & 0 & -1 & 0 & 1 & 0 & 0 & -1 \\
0 & -1 & 1 & 0 & 0 & 1 & 0 & 0 & -1 & 1 & 0 \\
0 & 0 & -1 & 1 & -1 & 0 & 0 & 0 & 0 & -1 & 1 \\
\end{array} \right\rgroup.\]

Since $\per{A_D} = -4 \neq 0$, we have $\pind(A_D) = 1$ (each column from $A_D$ is chosen once).  This implies, by Lemma \ref{mindtopind}, that $\mind(G) = 1$ (i.e. the term $x_1x_2x_3x_4x_5x_6$ has nonzero coefficient).  Hence, $\chSe(G) \leq 2$ by the Combinatorial Nullstellensatz.  Since there are adjacent vertices of equal degree in $G$, we have $\Se(G) \neq 1$, implying $\chSe(G) \neq 1$ and so $\chSe(G) = 2$.  Furthermore, since no column of $B_D$ was needed to find a submatrix of $M_D$ having nonzero permanent, we conclude that we may assign lists of size $1$ to each vertex as well and so $G$ is $(1,2)$-weight choosable.


%



\section{Some intermediary results on permanent indices and monomial indices}

The major results of this paper are proven by establishing bounds on $\mind(G)$ and $\tmind(G)$ using the permanent method.  One important tool is the following lemma, a generalization of a similar result in \cite{BGN09}:

\begin{lem}[Przyby{\l}o, Wo\'zniak \cite{PW11}]\label{lincomb}
Let $A$ be an $m \times l$ matrix, and let $L$ be an $m \times m$ matrix where each column of $L$ is a linear combination of columns of $A$.  Let $n_j$ denote the number of columns of $L$ in which the $j^{\textup{th}}$ column of $A$ appears with nonzero coefficient.  If $\per{L} \neq 0$, then \mbox{$\pind(A) \leq \textup{max}\,\{n_j \,|\, j = 1, \ldots l\}$}.
\end{lem}

We will also find the following theorem useful, which gives a method for constructing graphs in a way that preserves the property of having low monomial index:

\begin{thm}[Bartnicki, Grytczuk, Niwczyk \cite{BGN09}]\label{addtwins}
Let $G$ be a simple graph with $\mind(G) \leq 2$.  Let $U$ be a nonempty subset of $V(G)$.  If $F$ is a graph obtained by adding two new vertices $u,v$ to $V(G)$ and joining them to each vertex of $U$, and $H$ is a graph obtained from $F$ by joining $u$ and $v$, then $\mind(F), \mind(H) \leq 2$.
\end{thm}

As a consequence, the following graph classes have low monomial index and hence small values of $\chSe(G)$:

\begin{cor}[Bartnicki, Grytczuk, Niwczyk \cite{BGN09}]\label{twincor}
If $G$ is a complete graph, a complete bipartite graph, or tree, then $\mind(G) \leq 2$ and hence  $\chSe(G) \leq 3$.
\end{cor}

It can also be easily shown that the same bound holds for cycles.

\begin{prop}\label{cyclemind}
If $G = C_n$, then $\mind(G) \leq 2$ and hence $\chSe(G) \leq 3$.
\end{prop}

\begin{proof}
Let $V(G) = \{v_1, v_2, \ldots, v_n\}$ and $E(G) = \{v_1v_2, v_2v_3, \ldots, v_{n-1}v_n, v_nv_1\}$.  Let $D$ be the orientation of $G$ with $A(D) = \{(v_1,v_2), (v_2,v_3), \ldots, (v_{n-1},v_n), (v_n,v_1)\}$.  Consider the colouring polynomial 
\[
P_D = (x_2 - x_n)(x_3 - x_1)(x_4 - x_2)\cdots(x_{n} - x_{n-2})(x_{1} - x_{n-1}).
\]
Since each variable appears in exactly two factors of $P_D$, no exponent in the expansion of $P_D$ exceeds $2$, and hence $\mind(G) \leq 2$.
\end{proof}

In order to prove our major results in Section \ref{ch:alon:edge}, the following generalization of Theorem \ref{addtwins} is required:

\begin{lem}\label{twins}
Let $G$ be a graph with finite monomial index $\mind(G) \geq 1$.  Let $U$ be a nonempty subset of $V(G)$.  If $F$ is a graph obtained by adding two new vertices $u,v$ to $V(G)$ and joining them to each vertex of $U$, and $F^*$ is a graph obtained from $F$ by joining $u$ and $v$, then $\mind(F), \mind(F^*) \leq \max\{2, \mind(G)\}$.
\end{lem}

The proof which follows is an adaptation of the proof of Theorem \ref{addtwins} in \cite{BGN09}.  Given a matrix $A$ with columns $a_1, a_2, \ldots, a_n$ and a sequence of (not necessarily distinct) column indices $K = (i_1, i_2, \ldots, i_k)$, $A(K)$ is defined to be the matrix $A(K) = \left( a_{i_1} \,\, a_{i_2}  \,\, \cdots \,\, a_{i_k} \right).$

\begin{proof}
Let $U = \{u_1, \ldots, u_k\}$ be the subset of $V(G)$ stated in the theorem.  Let $E_u = \{e_1, e_3, \ldots, e_{2k-1}\}$ and $E_v =  \{e_2, e_4, \ldots, e_{2k}\}$ be the sets of edges incident to the vertices $u$ and $v$, respectively.  Assume that these edges are oriented toward $U$, and that for each $i = 1, 2, \ldots, k$ the edges $e_{2i-1}$ and $e_{2i}$ have the same head.

Let $D$ be an orientation of $F$, $D'$ the induced orientation of $G$, and consider the matrices $A_D$ and $A_{D'}$.
Let $A_1, \ldots, A_{2k}$ be the first $2k$ columns of $A_D$, corresponding to $\{e_1, e_2, \ldots, e_{2k}\}$.  If we write $A = \left(A_1\,\, \cdots \,\, A_{2k}\right)$, then $A_D = \left( A \,\, B \right)$ where $B = \left(\begin{smallmatrix} X \\ A_{D'} \end{smallmatrix} \right)$.  

Let $Y$ be the $(2k) \times (2k)$ matrix and $Z$ the $(|E(F)| - 2k) \times (2k)$ matrix such that $A = \left[\begin{smallmatrix} Y \\ Z \end{smallmatrix} \right]$.
Since the edges $e_{2i-1}$ and $e_{2i}$ have the same head for each $i = 1, 2, \ldots, k$, $A_{2i-1}$ and $A_{2i}$ agree on $Z$.  Furthermore, $Y$ may be written as a block matrix, where $\left(\begin{smallmatrix} 0 & 1 \\ 1 & 0 \end{smallmatrix} \right)$ occupies the diagonals and $\left(\begin{smallmatrix} -1 & 0 \\ 0 & -1 \end{smallmatrix} \right)$ is everywhere else, as seen in Figure \ref{Y}.

\begin{figure}[h!]
\begin{center}
\[Y = \begin{pmatrix}
0 & 1 & -1 & 0 & & -1 & 0 \\
1 & 0 & 0 & -1 & \cdots & 0 & -1 \\
-1 & 0 & 0 & 1 & & -1 & 0 \\
0 & -1 & 1 & 0 & & 0 & -1 \\
& \vdots & & & \ddots & & \\
-1 & 0 & -1 & 0 & & 0 & 1 \\
0 & -1 & 0 & -1 & & 1 & 0 \\
\end{pmatrix}.\]
\end{center}
\caption{The block matrix $Y$}\label{Y}
\end{figure}

There exists a matrix of columns from $A_{D'}$, with no column used more than $\mind(G)$ times, with nonzero permanent.  Let $K$ denote the sequence of edges of $G$ which index this matrix.  Consider a new matrix 
$$M = \left(M_1\,\, M_1  \,\,M_2\,\, M_2\,\,  \cdots \,\,M_k \,\,M_k \,\,B(K)\right),$$ 
where $M_j = A_{2j-1} - A_{2j}$ for $j = 1, 2, \ldots, k$.

The properties of the columns of $A$ outlined above imply that the matrix $M$ can be written as follows:
$M = \left( \begin{smallmatrix}
R & X(K) \\
0 & A_{D'}(K) 
\end{smallmatrix} \right)$, 
where $R$ has all constant rows:
$$R = \begin{pmatrix}
-1 & -1 & -1 & & -1 \\
1 & 1 & 1 & \cdots & 1 \\
-1 & -1 & -1 & & -1 \\
& \vdots & & \ddots & \\
1 & 1 & 1 & & 1 \\
\end{pmatrix}.$$

Since $\per{M} = \per{R} \times \per{A_{D'}(K)}$, each of $\per{R}$ and $ \per{A_G(K)}$ are nonzero, and any column of $A$ appears in the linear combination of at most 2 columns of $M$, Lemma \ref{lincomb} implies that $\mind(F) \leq \max\{2, \mind(G)\}$.

We now consider $F^*$.
Let $H$ be an orientation of $F^*$ with $e_0 = uv$ oriented from $v$ to $u$.
The matrix $A_H$ is precisely $A_D$ with a row and column added for $e_0$ (say, as the first row and column).  It can be depicted in block form
$A_H = \left(\begin{smallmatrix} Y' & X' \\
Z' & A_G  \end{smallmatrix} \right)$,
where $Y'$ and $Z'$ are the matrices depicted in Figure \ref{Y'Z'} on page \pageref{Y'Z'}.

\begin{figure}[h!]
\begin{center}
\[Y' = \begin{pmatrix}
0 & 1 & -1 & \cdots & 1 & -1 \\
-1 &&&&& \\
-1 &&&&& \\
\vdots &&&Y&& \\
-1 &&&&& \\
-1 &&&&& \\
\end{pmatrix}, \,\,\,\,
Z' = \begin{pmatrix}
0 & \\
0& \\
\vdots& Z \\
0& \\
0& \\
\end{pmatrix}\]
\end{center}
\caption{The matrices $Y'$ and $Z'$}\label{Y'Z'}
\end{figure}

Let $A_0, A_1, \ldots, A_{2k}$ denote the first $2k+1$ columns of $A_H$, corresponding to the edges $e_0, e_1, \ldots, e_{2k}$.  Form a new matrix
$$N = \left(N_0\,\, N_0\,\, N_1  \,\,N_2\,\, N_2\,\,  \cdots \,\,N_k \,\,N_k \,\,B(K)\right),$$
so that $N_0 = A_0$ and $N_j = A_{2j-1} - A_{2j}$ for $j = 1, 2, \ldots, k$.  Arguing as before, 
$N = \left(\begin{smallmatrix}
R' & X'(K) \\
0 & A_G(K) \\
\end{smallmatrix}\right)$,
where $R'$ is the following square matrix:
$$R' = \begin{pmatrix}
0 & 0 & 2 & 2 & & 2 \\
-1 & -1 & -1 & -1 & \cdots & -1 \\
-1 & -1 & 1 & 1 & & 1 \\
& \vdots & \vdots & & \ddots & \\
-1 & -1 & 1 & 1 & & 1 \\
\end{pmatrix}.$$
It is shown in \cite{BGN09} that $\per{R'} \neq 0$.  Hence $\per{N} = \per{R'}\times\per{A_G(K)} \neq 0$, and since any column of $A$ appears in the linear combination of at most 2 columns of $N$, Lemma \ref{lincomb} implies thats $\mind(H) \leq \max\{2, \mind(G)\}$.
\end{proof}

\section{General bounds for edge list-weightings and total list-weightings}\label{ch:alon:edge}

Armed with the Combinatorial Nullstellensatz and the permanent method, we may now proceed with our main theorems.

Recall that a graph $G$ is {\em $d$-degenerate} if every induced subgraph of $G$ has a vertex of degree at most $d$.  If $G$ and $H$ are graphs, we write $H \leq_i G$ to denote that $H$ is an induced subgraph of $G$.  
The degeneracy of a graph $G$, denoted $\partial(G)$, is the smallest $d$ for which $G$ is $d$-degenerate; that is $\partial(G) = \max\{\delta(H) | H \leq_i G\}$.
We extend the notion of degeneracy to pairs of vertices at a given distance.  
Let $\delta_t(G)$ denote the minimum value of $d(u) + d(v)$ for two vertices $u,v \in V(G)$ at distance exactly $t$.
The {\em $t$-degeneracy} of a graph $G$, denoted $\partial_t(G)$, is
\[
\partial_t(G) = \max\{\delta_t(H) | H \leq_i G\}.
\]
We say that $G$ is {\em $(t,d)$-degenerate} for each integer $d \geq \partial_t(G)$.  If no induced subgraph of $G$ has vertices at distance exactly $t$ (for example, $G = K_n$ and $t \geq 2$), then we adopt the convention that $\partial_t(G) = 2\Delta(G)$.

We now show that $\chSe(G)$, $\chPe(G)$, and $\chSt(G)$ can be bounded by functions of $\partial_2(G)$.  
As a corollary, each parameter can in turn be bounded in terms of the maximum degree and degeneracy of $G$.  
The results are achieved by carefully orienting the edges of a graph and applying the lemmas from the previous sections to show that our desired matrix has non-zero permanent.





%





\begin{thm}\label{edgemain}
If $G$ is a nice graph on at least $3$ vertices, then $\mind(G) \leq \partial_2(G)$ and hence $\chSe(G) \leq \partial_2(G) + 1$
\end{thm}

\begin{proof}
We need only prove that $\mind(G) \leq \partial_2(G)$; the rest of the theorem follows from Lemma \ref{mind}.
If $G$ is a tree, cycle, or complete graph, then $\mind(G) \leq 2$ by Corollary \ref{twincor} and Proposition \ref{cyclemind}, and hence the theorem holds for the following graphs: $P_3$, $K_3$, $P_4$, $K_{1,3}$, $C_4$, and $K_4$.  If $G$ is isomorphic to $K_3$ with a leaf or $C_4$ with a chord, then one may check that the theorem holds for $G$ by straightforward computation of the associated colouring polynomial $P_D$ for any orientation $D$.  Hence, the theorem holds for any connected graph on $3$ or $4$ vertices.

We proceed now by induction on $|V(G)|$.  
We may assume that $G$ is connected, since Proposition \ref{componentbound} states that $\mind(G)$ is at most the largest monomial index of its components.  
Let $G$ be a connected graph on at least $5$ vertices, and for any graph $H$ with $|V(H)| < |V(G)|$, assume that $\mind(H) \leq \partial_2(H)$.  


If $G$ is a complete graph, then the theorem holds by Corollary \ref{twincor}.  Assume that $G$ is not complete.  There exist $u,v,w \in V(G)$ such that the induced subgraph $G[\{u,v,w\}]$ is a path of length $2$ (or, $uvw$ is an induced $2$-path).  
Choose this $2$-path such that $d(u)+d(w)$ is minimum (and, hence, $d(u) + d(w) \leq \partial_2(G))$.  The ultimate goal will be to apply an inductive argument to $G - \{u,w\}$, however we must concern ourselves with whether or not this subgraph of $G$ is nice.  To this end, we define the following sets of edges:
\begin{align*}
\mathcal{F} &= \textrm{ the edges of those components in $G - \{u,w\}$ isomorphic to $K_2$} \\
E_u &= \{e \in E(G) \,|\, e \ni u, e \neq uv\} \\
E_w &= \{e \in E(G) \,|\, e \ni w, e \neq vw\} \\
E_v &= \{e \in E(G) \,|\, e \ni v, e \neq uv,vw\} \\
E^* &= E(G) \setminus \left(E_u \cup E_v \cup E_w \cup\{uv, vw\}\right)
\end{align*}
The path $uvw$ and the sets of edges $E_u,E_v,E_w$ are shown in Figure \ref{path}. \\


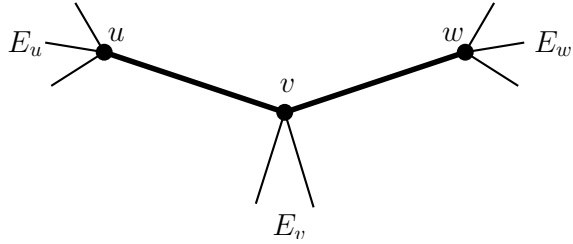
\begin{figure}[h!]
\begin{center}
\scalebox{0.8}{
\begin{tikzpicture}
\clip(-5,-2.1) rectangle (5,2.5);
\draw[line width=2.6pt] (0,0)-- (-3,1);
\draw[line width=2.6pt] (0,0)-- (3,1);
\draw [line width=1pt] (-3.48,1.82)-- (-3,1);
\draw [line width= 1pt] (-3.98,1.16)-- (-3,1);
\draw [line width= 1pt] (-3.88,0.44)-- (-3,1);
\draw [line width= 1pt] (3.48,1.82)-- (3,1);
\draw [line width= 1pt] (3.98,1.16)-- (3,1);
\draw [line width= 1pt] (3.88,0.44)-- (3,1);
\draw [line width= 1pt] (-0.48,-1.52)-- (0,0);
\draw [line width= 1pt] (0.48,-1.58)-- (0,0);
\draw (-4.75,1.5) node[anchor=north west] {\large $E_u$};
\draw (-0.35,-1.55) node[anchor=north west] {\large $E_v$};
\draw (4,1.5) node[anchor=north west] {\large $E_w$};
\fill [color=black] (0,0) circle (4pt);
\draw[color=black] (0.05,0.45) node {\large $v$};
\fill [color=black] (-3,1) circle (4pt);
\draw[color=black] (-2.8,1.3) node {\large $u$};
\fill [color=black] (3,1) circle (4pt);
\draw[color=black] (2.8,1.3) node {\large $w$};
\end{tikzpicture}
}
\end{center}
\caption{The induced $2$-path $uvw$ in $G$}\label{path}
\end{figure}

\noindent{\underline{Case 1: $E_v \cap \mathcal{F} \neq \emptyset$}}

If $E_v \cap \mathcal{F} \neq \emptyset$, then there can be only one edge in this intersection, otherwise the connected component containing $v$ in $G - \{u,w\}$ would have two or more edges. This implies that, since $uv, vw \in E(G)$, we have that $N_G(v) = \{u,w,x\}$ for some vertex $x \in V(G)$.  Since $\{v,x\}$ induces a graph isomorphic to $K_2$ in $G - \{u,w\}$, we have that $N_G(x) \subseteq \{u,v,w\}$.  

If $x$ is adjacent to both $u$ and $w$, then $v$ and $x$ are adjacent twins.  Suppose that $G \setminus \{v,x\}$ is not nice; we will show that this contradicts the choice of $uvw$ which minimizes $d_G(u) = d_G(w)$.  If $G$ is not nice, then $u$, $w$, or both $u$ and $w$ are adjacent to exactly one vertex in $G$ other than $v$ and $x$; without loss of generality, suppose that $uy \in E(G)$, $y \neq v,x$.  Since $y \notin N_G(x)$, the vertices $y,u,x$ induce a $2$-path; furthermore, $d_G(y) + d_G(x) = 1+ 3 = 4$.  This contracts our choice of $uvw$, since $d(u) + d(w) \geq 3 + 2 = 5$.  Thus, $G - \{v,x\}$ is a nice graph, and so, by Lemma \ref{twins}, $\mind(G) \leq \max\{2,\mind(G - \{v,x\})\}$.  By the induction hypothesis, $\mind(G - \{v,x\}) \leq \partial_2(G - \{v,x\})$, and so
\[ 
\mind(G) \leq \max\{2,\mind(G - \{v,x\})\} \leq \max\{2,\partial_2(G - \{v,x\})\} \leq \partial_2(G).
\]

We may now assume that $x$ is not adjacent to at least one of $u$ and $w$.  If $w \notin N_G(x)$, then both $uvw$ and $xvw$ are induced $2$-paths in $G$.  By the minimality of $d(u) + d(w)$, we must have that $d(u) \leq d(x)$.  If $u$ is adjacent to $x$, then $d(x) = 2$ and, since $u$ is adjacent to $v$ as well, $d(u) = 2$ and $N_G(u) = \{v,x\}$.  Otherwise, if $u$ is not adjacent to $x$, then $d_G(u) = d_G(x) = 1$ and $N_G(u) = \{v\}$.  In either case, $u$ and $x$ are twins.  If $G -\{u,x\}$ is not nice, then the only edge not incident to $u$ or $x$ is the edge $vw$, contradicting our choice of $G$ with $|V(G)| \geq 5$.  
Assume that $G - \{u,x\}$ is nice.  By Lemma \ref{twins}, $\mind(G) \leq \max\{2,\mind(G - \{u,x\})\}$, and by the induction hypothesis, $\mind(G - \{u,x\}) \leq \partial_2(G - \{u,x\}))$.  Thus,
\[ 
\mind(G) \leq \max\{2,\mind(G - \{u,x\})\} \leq \max\{2,\partial_2(G - \{u,x\})\} \leq \partial_2(G).
\]
If $u \notin N_G(x)$ and $w \in N_G(x)$, then the exact same argument holds as for $u \in N_G(x)$ and $w \notin N_G(x)$.  Having considered all possible neighbourhoods of $x$, we conclude that if $E_v \cap \mathcal{F}$ is nonempty, then $\mind(G) \leq \partial_2(G)$. \\


\noindent{\underline{Case 2: $E_v \cap \mathcal{F} = \emptyset$}}

Suppose that $E_v \cap \mathcal{F} =\emptyset$.  The argument proceeds as follows:  after choosing a ``good'' orientation $D$ of $G$, we will construct a matrix whose columns are linear combinations of $A_D$ with no column of $A_D$ being used more than $\partial_2(G)$ times and with nonzero permanent.  The result will then follow by Lemma \ref{lincomb}.

Let $D$ be an orientation of $G$ where the edges of $E_u \cup \{uv\}$ and $E_v$ are oriented toward $u$ and $v$, respectively, and the edges of $E_w \cup \{vw\}$ are oriented away from $w$
; see Figure \ref{orientedpath}.
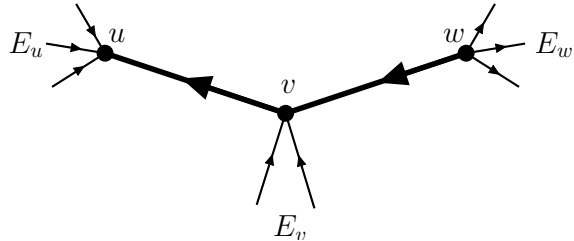
\begin{figure}[h!]
\begin{center}
\scalebox{0.8}{
\begin{tikzpicture}
\clip(-5,-2.1) rectangle (5,2.5);
\draw[line width=2.6pt] (0,0)-- (-3,1);
\draw[line width=2.6pt] (0,0)-- (3,1);
\draw [line width=1pt] (-3.48,1.82)-- (-3,1);
\draw [line width= 1pt] (-3.98,1.16)-- (-3,1);
\draw [line width= 1pt] (-3.88,0.44)-- (-3,1);
\draw [line width= 1pt] (3.48,1.82)-- (3,1);
\draw [line width= 1pt] (3.98,1.16)-- (3,1);
\draw [line width= 1pt] (3.88,0.44)-- (3,1);
\draw [line width= 1pt] (-0.48,-1.52)-- (0,0);
\draw [line width= 1pt] (0.48,-1.58)-- (0,0);
\draw [-triangle 45, line width=2.3pt] (0,0) -- (-1.73,0.58);
\draw [-triangle 45, line width=2.3pt] (3,1) -- (1.53,0.51);
\draw [-triangle 45, line width=0.0pt] (-3.46,1.78) -- (-3.16,1.28);
\draw [-triangle 45, line width=0.0pt] (-3.98,1.16) -- (-3.4,1.07);
\draw [-triangle 45, line width=0.0pt] (-3.88,0.44) -- (-3.35,0.78);
\draw [-triangle 45, line width=0.0pt] (-0.48,-1.52) -- (-0.22,-0.69);
\draw [-triangle 45, line width=0.0pt] (0.47,-1.56) -- (0.22,-0.71);
\draw [-triangle 45, line width=0.0pt] (3,1) -- (3.29,1.5);
\draw [-triangle 45, line width=0.0pt] (3,1) -- (3.66,1.11);
\draw [-triangle 45, line width=0.0pt] (3,1) -- (3.55,0.65);
\draw (-4.75,1.5) node[anchor=north west] {\large $E_u$};
\draw (-0.35,-1.55) node[anchor=north west] {\large $E_v$};
\draw (4,1.5) node[anchor=north west] {\large $E_w$};
\fill [color=black] (0,0) circle (4pt);
\draw[color=black] (0.05,0.45) node {\large $v$};
\fill [color=black] (-3,1) circle (4pt);
\draw[color=black] (-2.8,1.3) node {\large $u$};
\fill [color=black] (3,1) circle (4pt);
\draw[color=black] (2.8,1.3) node {\large $w$};
\end{tikzpicture}
}
\end{center}
\caption{An orientation $D$ of a graph $G$ with an induced $2$-path $uvw$}\label{orientedpath}
\end{figure}
Let $c_{uv}$ and $c_{vw}$ be the columns of $A_D$ associated with the edges $uv$ and $vw$, respectively, and let $c = c_{uv} - c_{vw}$; see Figure \ref{columnoperations}.
\begin{figure}[h!]
{ \footnotesize
\begin{eqnarray*}
\bordermatrix{~ & c_{uv} & c_{vw} & \cdots \cr
                  uv & 0 & -1 & \cr
                  vw & 1 & 0 & \cr
                  \cr
                    & 1 & 0 & \cr
                 E_u  & \vdots & \vdots & \cr
                    & 1 & 0 & \cr
                  \cr
                    & 0 & -1 & \cr
                 E_w & \vdots & \vdots & \cdots \cr
                    & 0 & -1 & \cr
                  \cr
                    & 1 & 1 & \cr
                 E_v  & \vdots & \vdots & \cr
                    & 1 & 1 & \cr
                  \cr
                    & 0 & 0 & \cr
                 E^*  & \vdots & \vdots & \cr
                    & 0 & 0 & \cr}
     &\implies& 
    c = \bordermatrix{~  \cr
                  & 1  \cr
                  & 1  \cr
                  \cr
                    & 1  \cr
                 & \vdots  \cr
                    & 1  \cr
                  \cr
                    &  1  \cr
                & \vdots  \cr
                    &  1  \cr
                  \cr
                    & 0  \cr
                  & \vdots   \cr
                    & 0 \cr
                  \cr
                    & 0  \cr
                  & \vdots   \cr
                    & 0   \cr}
\end{eqnarray*}}
\caption{An operation on two columns of $A_D$}\label{columnoperations}
\end{figure}

We must still concern ourselves with the possibility that deleting $u$ and $w$ from $G$ gives a graph which is not nice.  If a component of $G - \{u,w\}$ is isomorphic to $K_2$, then one vertex of this component must be adjacent to either $u$ or $w$ in $G$.  Let $\mathcal{F} = \{f_1, \ldots, f_k\}$.  For each $f_i \in \mathcal{F}$, let $e_i$ be an edge from $E_u$ or $E_w$ to which $f_i$ is adjacent.  Let $F$ denote this collection of edges from $E_u \cup E_w$, and let $F_u = \{e \,:\, e \in E_u \cap F\}$ and $F_w = \{e \,:\, e \in E_w\cap F\}$.  Each edge $f_i \in \mathcal{F}$ will be oriented away from its shared endpoint with $e_i$.  

Let $H = G - \{u,w\} - \mathcal{F}$ and $D(H)$ be the corresponding sub-digraph of $D$.  Since we have removed all components isomorphic to $K_2$, $H$ is nice.  Since $H$ has fewer vertices than $G$, by the induction hypothesis, $\mind(H) \leq \partial_2(H)$.  Hence, there exists a matrix $L_H$ consisting of columns of $A_{D(H)}$, 
none repeated more than $\partial_2(H)$ times, with $\per(L_H) \neq 0$.  Let $K$ denote the sequence of edges which indexes the columns of $L_H$.
Recall that, for an $m \times n$ matrix $A$, $A^{(k)}$ is the $m \times kn$ matrix consisting of $k$ consecutive copies of $A$ (see page \pageref{kcopydefn}).  Let $L_G$ be the following block matrix: 
\[
L_G = \Big( \,\,c^{\left(d(u) + d(w)\right)} \,\,\Big|\,\,  A_D(F) \,\,\Big|\,\, A_D(K)\,\, \Big) = \bordermatrix{ ~ & ~ & ~& ~ \cr
{\scriptstyle E_u \cup E_w \cup \{uv,vw\}} & J_{d(u) + d(w)} & K_1 & X_1 \cr
 \hfill {\scriptstyle \mathcal{F} }& 0 & K_2 & X_2 \cr
 \hfill {\scriptstyle E(H)} & 0 & 0 & L_H \cr
},
\]
where the blocks are as follows:
	\begin{itemize}
	\item $J_{d(u) + d(w)}$ is the $\left(d(u) + d(w)\right) \times \left(d(u) + d(w)\right)$ all $1$'s matrix.
	\item $K = \left(\begin{smallmatrix} K_1 \\ K_2 \end{smallmatrix} \right)$ having entries depending on whether the column is indexed by $e_i \in F_u$ or $e_i \in F_w$.  If the column is indexed by $e_i \in F_u$, then the column will have (i) $1$ in each row indexed by the other edges from $E_u$, (ii) $1$ in the row indexed by $uv$, (iii) $-1$ in the row indexed by $f_i$, and (iv) $0$ in all other entries.
Otherwise, the entries follow the same pattern with the signs swapped.  Since the column associated with $e_i$ has only one non-zero entry in the rows indexed by $\mathcal{F}$, $K_2$ is diagonal with $|F_u|$ entries being $-1$ and $|F_w|$ entries being $1$.
	\item $X = \left(\begin{smallmatrix} X_1 \\ X_2\end{smallmatrix} \right)$, the $\left(|E(G)|-|E(H)|\right) \times |E(H)|$ submatrix of $A_D(K)$ whose rows are indexed by $E(G) \setminus E(H)$; and
	\item $L_H$, is the matrix with $\per(L_H) \neq 0$ defined above.
	\end{itemize}

Since $J_{d(u) + d(w)}$, $K_2$, and $L_H$ are all square matrices,  
\begin{eqnarray*}
\per(L_G) &=& \per\left(J_{d(u) + d(w)}\right)\cdot\per(K_2)\cdot\per(L_H) \\
&=& \left(d(u) + d(w)\right)!\cdot \left(-1\right)^{|F_u|} \left(1\right)^{|F_w|}\cdot\per(L_H) \neq 0.
\end{eqnarray*}
Since the sets $\{uv, vw\}$, $F$, and $E(H)$ are pairwise disjoint, no column is used more than $\max\{d(u)+d(w), 1, \mind(H)\}$ times.  
Lemma \ref{lincomb} states that $\pind(A_D) \leq \max\{d(u)+d(w), 1, \mind(H)\}$.  Since $\pind(A_D) = \mind(G)$ (Lemma \ref{mindtopind}.1) and $\mind(H) \leq \partial_2(H)$ by induction,
\[
\mind(G) \leq \max\{d(u)+d(w), 1, \mind(H)\} \leq \max\{\partial_2(G), 1, \partial_2(H)\} \leq \partial_2(G). \qedhere
\] 
\end{proof}

In the proof of Theorem \ref{edgemain}, we may think of $A_D$ as a submatrix of $M_D$, and so we are constructing a matrix with nonzero permanent using no column corresponding to a vertex of the graph.  It follows that one may assign arbitrary lists of size $1$ to each vertex; in other words Theorem \ref{edgemain} implies the following results on $(k,l)$-weight choosability:

\begin{cor}\label{chooseedge}
If $G$ is a nice graph, then it is $(1, \partial_2(G) + 1)$-weight choosable.
\end{cor}

We now present an immediate corollary to Theorems \ref{edgemain}, which follows from the following simple observation:

\begin{prop}
If $G$ is a $d$-degenerate graph with maximum degree $\Delta(G)$, then $\partial_2(G) \leq \Delta(G) + d$.
\end{prop}

\begin{proof}
Let $H$ be an induced subgraph of $G$ and let $v \in V(H)$ have degree $d(v) \leq d$.  If every vertex of $H$ is adjacent to $v$, then $H$ must be a clique and any pair of vertices have degrees summing to at most $2d \leq \Delta(G) + d$.  Otherwise, there exists some vertex at distance $2$ from $v$, and these two vertices have degrees that sum to at most $\Delta(G) + d$.
\end{proof}

\begin{cor}
If $G$ is a nice $d$-degenerate graph with maximum degree $\Delta(G)$, then $G$ is $(1, \Delta(G) + d + 1)$-weight choosable or, alternately, $(1, \Delta(G) + \col(G))$-weight choosable.
\end{cor}

\section{Weight choosability of graph products}\label{more}

Though the bounds presented in Section \ref{ch:alon:edge} and the corollaries above represent progress on the List 1-2-3 Conjecture, there is still room for improvement.  We now consider some classes of graphs where smaller upper bounds can be obtained, in particular the cartesian product of two graphs.  The following decomposition lemma on $\mind(G)$ provides an approach for such graphs:

\begin{lem}\label{decomp}
Let $G$ be a graph, and let $H$ be an induced subgraph of $G$ containing a $2$-factor.  Let $X$ be a minimal edge cut separating $V(H)$ from $V(G) \setminus V(H)$.  If the components of $G - H - X$ are $C_1, \ldots, C_k$, then 
\begin{eqnarray*}
\mind(G) &\leq& \max\{|X| + \mind(H), \mind(G-X)\} \\
 &=& \max\{\mind(H) + |X|, \mind(C_1), \ldots,  \mind(C_k)\}.
\end{eqnarray*}
\end{lem}

\begin{proof}
Let $|V(H)| = v$ and $F = \{e_1, \ldots, e_v\}$ be a $2$-factor of $H$.  Let $D$ be an orientation of $G$ such that the cycles of $F$ are directed.  Define the column vector $c = \sum_{i=1}^v c_i$ where $c_i$ is the column of $A_D$ corresponding to $e_i$.  For each $e \in E(H) \setminus F$ there are two edges of $F$ incident to each of the head and tail of $e$, and for each $e \in F$ there is one edge of $F$ incident to each of the head and tail of $e$.  Hence, the entries of $c$ are nonzero in the rows indexed by the edges of $X$ and $0$ in all other entries.

There exists a matrix $L_{G-X}$ consisting of columns of $A_{G-X}$ with no column of $A_D$ repeated more than $\mind(G - X)$ times and $\per(L_{G-X}) \neq 0$.  Let $K$ denote the sequence of edges of $G - X$ which index $A_{G-X}$.  Consider the following matrix:
\begin{eqnarray*}
L = \Big(c^{\left(|X|\right)} \,\,\,\, A_D(K) \Big)  = \begin{pmatrix}
 M & N \\
 0 & L_{G-X}
  \end{pmatrix},
\end{eqnarray*}
where $\left( M\,\, N\right)$ is indexed by $X$, each row of $M$ is constant, and every entry of $M$ is nonzero.  Any column indexed by $e \in E(G) \setminus F$ is used at most $\mind(G-X)$ times in the construction of $L$, and any edge from $F$ is used at most $|X| + \mind(H)$ times.  Clearly, $\per(L) = \per(M)\per(L_G-X) \neq 0$, and hence $\mind(G) = \max\{|X| + \mind(H), \mind(G-X)\}$.  Since $\mind(G-X) = \max\{ \mind(C_1), \ldots,  \mind(C_k), \mind(H)\}$ by Proposition \ref{componentbound}, the result follows.
\end{proof}


Recall that the Cartesian product of two graphs $G$ and $H$, denoted by $G\,\Box\,H$, is defined as the graph having vertex set $V(G) \times V(H)$ where two vertices $(u,u')$ and $(v,v')$ are adjacent if and only if either $u = v$ and $u'$ is adjacent to $v'$ in $H$ or $u' = v'$ and $u$ is adjacent to $v$ in $G$.  Some results on $\Se(G)$ for Cartesian products of graphs are given in \cite{Ben1}; for instance, if $G$ and $H$ are regular and bipartite, then $\Se(K_n \,\Box\ H)$, $\Se(C_t \,\Box\ H)$, and $\Se(G \,\Box\ H)$ are at most $2$ for $n \geq 4$, and $t \geq 4, t \neq 5$.  Lemma \ref{decomp} may be used to bound $\chSe(G \,\Box\, H)$ for many more graphs $G$ and $H$.
Note that, for the graph $G\,\Box\,H$ and vertex $v \in V(G)$, the subgraph induced by the set of vertices $\{(v,x) \,:\, x \in V(H)\}$ is denoted $(v,H)$.





\begin{thm}
Let $H$ be a regular graph on $n \geq 3$ vertices which contains a 2-factor.  If $G$ is a $d$-degenerate graph, then \textup{(1)} $\mind(G\,\Box\,H) \leq nd + \mind(H)$ and \textup{(2)} $\chSe(G\,\Box\,H) \leq nd + \mind(H) + 1.$
\end{thm}

\begin{proof}
We may assume that $G$ is connected.  The proof of (1) is by induction on $|V(G)|$; the statement is true when $G$ is a single vertex, since $d= 0$ and $\chSe(H) \leq \mind(H)+1$ is guaranteed by Lemma \ref{mind}.

Suppose $|V(G)| \geq 2$.  Let $v \in V(G)$ have degree at most $d$, and let $X$ be the minimal edge cut for $(v,H)$.  Since $|X| =n\cdot d_G(v)$ and $G-v$ is $d$-degenerate, Lemma \ref{decomp} implies that
\begin{eqnarray*} 
\mind(G\,\Box\,H) &\leq& \max\{\mind(H) + nd_G(v), \mind((G\,\Box\,H)-X)\} \\
&\leq& \max\{\mind(H) + nd, \mind((G\,\Box\,H) - (v,H)\} \\
&\leq& \max\{\mind(H) + nd, \mind((G-v)\,\Box\,H)\} \\
&\leq& \max\{\mind(H) + nd, nd + \mind(H) \} \\
&\leq& \mind(H) + nd.
\end{eqnarray*}
Part (2) follows directly from Lemma \ref{mind}.
\end{proof}

Since $\mind(K_n)$ is at most 2 by Corollary \ref{twincor}, the following corollary is obtained:

\begin{cor}
For any integer $n \geq 3$ and any $d$-degenerate graph $G$, $\chSe(G \,\Box\, K_n) \leq nd + 3$.
\end{cor}

Since such a graph is $d(n-1)$-degenerate, this improves the bound of $\chSe(G \,\Box\, K_n) \leq 2d(n-1)$ given by Pan and Yang \cite{PY12}.


\section{Acknowledgements}
The majority of this work was completed as part of the author's doctoral thesis at Carleton University.  Thanks are due to the Natural Sciences and Engineering Research Council of Canada and to Carleton University for their financial support, and to Brett Stevens for his valuable supervisory input.  

\bibliographystyle{plain}
\bibliography{references}

\end{document}